\documentclass[11pt,reqno]{amsart}
\usepackage{amsmath,amsthm,amsfonts,amssymb,graphicx,mathtools,enumerate}
\usepackage[lmargin=35mm,rmargin=35mm,bmargin=30mm,tmargin=30mm]{geometry}
\usepackage[usenames,dvipsnames,svgnames,table]{xcolor}
\usepackage[foot]{amsaddr}
\usepackage[pdftitle={Separation Dimension and Degree},colorlinks=true,citecolor=black,linkcolor=black,urlcolor=blue!80!black]{hyperref}
\usepackage[noabbrev,capitalise]{cleveref}
\usepackage[numbers,sort&compress]{natbib}
\allowdisplaybreaks
\makeatletter
\def\NAT@spacechar{~}
\makeatother
\theoremstyle{plain}
\newtheorem{thm}{Theorem}
\newtheorem{lem}[thm]{Lemma}
\newtheorem{cor}[thm]{Corollary}

\crefname{lem}{Lemma}{Lemmas}
\crefname{thm}{Theorem}{Theorems}
\crefname{cor}{Corollary}{Corollaries}
\crefname{claim}{Claim}{Claims}
\newcommand{\arXiv}[1]{arXiv:\,\href{http://arxiv.org/abs/#1}{#1}}
\newcommand{\msn}[1]{MR:\,\href{http://www.ams.org/mathscinet-getitem?mr=MR#1}{#1}}
\newcommand{\doi}[1]{doi:\,\href{http://doi.org/#1}{#1}}
\DeclarePairedDelimiter\ceil\lceil\rceil
\DeclarePairedDelimiter\floor\lfloor\rfloor
\renewcommand{\geq}{\geqslant}
\renewcommand{\leq}{\leqslant}
\newcommand{\comment}[1]{\noindent\textcolor{red}{[#1]}}
\begin{document}

\title{Separation Dimension and Degree}

\author{Alex Scott}
\address[A. Scott]{Mathematical Institute, University of Oxford, Oxford OX2 6GG, United Kingdom}
\email{scott@maths.ox.ac.uk}
\thanks{Alex Scott is supported by a Leverhulme Research Fellowship.}

\author{David R. Wood}
\address[D. R. Wood]{School of Mathematical Sciences, Monash University, Melbourne, Australia}
\email{david.wood@monash.edu}


\date{\today}

\maketitle

\begin{abstract} 
The \emph{separation dimension} of a graph $G$ is the minimum positive integer $d$ for which there is an embedding of $G$ into $\mathbb{R}^d$, such that every pair of disjoint edges are separated by some axis-parallel hyperplane. We prove a conjecture of Alon~et~al.~[\emph{SIAM J. Discrete Math.} 2015] by showing that every graph with maximum degree $\Delta$ has separation dimension less than $20\Delta$, which is best possible up to a constant factor. We also prove that graphs with separation dimension 3 have bounded average degree and bounded chromatic number, partially resolving an open problem by Alon et~al.~[\emph{J. Graph Theory} 2018].  
\end{abstract}

\section{Introduction}

%The \emph{separation dimension} of a hypergraph $G$ is the minimum number of orderings of $V(G)$ such
%that for all disjoint edges $e$ and $f$ in $G$, in some ordering, all the vertices in $e$ appear before all the vertices in $f$ \citep{BCGMR14,BCGMR16,ABCMR15,LW18,BDL17,ZRMGD18}

This paper studies the separation dimension of graphs and its relationship with maximum and average degree. For a graph $G$, a function $f:V(G)\rightarrow \mathbb{R}^d$ is \emph{separating} if for all disjoint edges $vw,xy\in E(G)$ there is an axis-parallel hyperplane that separates the pair of points $\{f(v),f(w)\}$ from the pair $\{f(x),f(y)\}$. 
The \emph{separation dimension} of a graph $G$ is the minimum positive integer $d$ for which there is a $d$-dimensional separating function for $G$; see \citep{ABCMR18,BCGMR14,BCGMR16,ABCMR15,LW18,BDL17,ZRMGD18} for recent work on the separation dimension of graphs. 

This topic can also be thought of more combinatorially. Edges $e$ and $f$ in a  graph $G$ are \emph{separated} in a linear ordering of $V(G)$ if both endpoints of $e$ appear before both endpoints of $f$, or both endpoints of $f$ appear before both endpoints of $e$. A \emph{representation} of  $G$ is a non-empty set of linear orderings of $V(G)$.  A \emph{representation} $\mathcal{R}$ of $G$ is \emph{separating} if every pair of disjoint edges in $G$ are separated in at least one ordering in $\mathcal{R}$. It is easily seen that the separation dimension of $G$ equals the minimum size of a separating representation of $G$; see \citep{ABCMR18,BCGMR14,BCGMR16,ABCMR15,CMS11}.

A fundamental question is the relationship between separation dimension and maximum degree. \citet{CMS11} proved that every graph with maximum degree $\Delta$ has separation dimension at most $2\Delta (\ceil{\log_2 \log_2 \Delta} + 3) + 1$. \citet{ABCMR15} improved this bound to $2^{9\log^*(\Delta)}\Delta$, and conjectured that a stronger $O(\Delta)$ bound should hold.  We prove this  conjecture.

\begin{thm}
\label{MaxDegree}
Every graph with maximum degree $\Delta\geq 1$ has separation dimension less than $20\Delta$. 
\end{thm}

This linear bound is best possible up to a constant factor, since \citet{ABCMR15} proved that almost every $\Delta$-regular graph has separation dimension at least $\frac{\Delta}{2}$. \cref{MaxDegree} is proved in \cref{ProofMaxDegree}.

\cref{ProofAvgDeg} of this paper considers the following natural extremal question, first posed by \citet{ABCMR18}:  What is the maximum average degree of an $n$-vertex graph with separation dimension $s$? Every graph with separation dimension at most $2$ is planar, and thus has average degree less than $6$. For $s\geq 3$,  \citet{ABCMR18} proved the best  known upper bound on the average degree of $O(\log^{s-2} n)$, and asked %. \citet[Open~Problem~1.7]{ABCMR18}  asked 
whether graphs with bounded separation dimension have bounded degeneracy (or equivalently, bounded average degree). We answer the first open case of this problem.

\begin{thm}
\label{AvgDeg}
There is a constant $c$ such that every graph with separation dimension 3 has average degree at most $c$.
\end{thm}

%%%%%%%%%%%%%%%%%%%%%%%
\section{A Colouring Lemma}
\label{ColouringLemma}

%, the distinguishing feature of the following lemma is that $kd=(1+o(1))\Delta$ whenever $k\leq o(\frac{\Delta}{\log \Delta})$.

This section proves a straightforward lemma that shows how to colour a graph so that each vertex has few neighbours of each colour (\cref{Colour}). Several previous papers have proved similar results \citep{CMS11,FK86,SW18,KM11,MR10,HMR97}. The proof depends on the following two standard probabilistic tools. Let $[k]:=\{1,2,\dots,k\}$. 

\begin{lem}[Lov\'asz Local Lemma \citep{EL75}] 
\label{LLL}
Let $E_1,\dots,E_n$ be events in a probability space, each with probability at most $p$ and mutually independent of all but at most $D$ other events. If $4pD\leq 1$ then with positive probability, none of  $E_1,\dots,E_n$ occur. 
\end{lem}

\begin{lem}[Chernoff Bound \citep{MU05}]
\label{Chernoff}
Let $X_1,\dots,X_n$ be independent random variables, where $X_i=1$ with probability $p$ and $X_i=0$ with probability $1-p$. Let $X := \sum_{i=1}^n X_i$. Then for $\delta>0$, 
\begin{equation*}
\mathbb{P}( X \geq (1+ \delta)pn ) \leq e^{-\delta^2 pn/3}.
\end{equation*}
\end{lem}

\begin{lem}
\label{Colour}
For all positive integers $k$ and $\Delta$, 
for every graph $G$ with maximum degree at most $\Delta$, 
there is a partition $V_1,\dots,V_k$ of $V(G)$ such that for every vertex $v\in V(G)$ and integer $i\in[k]$, 
\begin{equation*}
|N_G(v)\cap V_i| < d := \frac{\Delta}{k} + \sqrt{\frac{3 \Delta\log(4k\Delta^2)}{k} }.
\end{equation*}
\end{lem}

\begin{proof}
Independently and randomly colour each vertex with one of $k$ colours. For each vertex $v\in V(G)$ and colour $c$, let $A_{v,c}$ be the event that at least $d$ neighbours of $v$ are all assigned colour $c$. Each event is mutually independent of all but at most $k\Delta^2$ other events. 

We now prove that $\mathbb{P}(A_{v,c}) \leq (4k\Delta^2)^{-1}$. 
Since $\mathbb{P}(A_{v,c})$ is increasing with $\deg(v)$, we may assume that $\deg(v)=\Delta$. 
Say $w_1,\dots,w_\Delta$ are the neighbours of $v$. 
For $i\in[\Delta]$, let $X_i:=1$ if $w_i$ is coloured $c$, otherwise let $X_i:=0$. 
Then $\mathbb{P}(X_i)=p :=\frac{1}{k}$. Let $X := \sum_{i=1}^\Delta X_i$. 
Then $A_{v,c}$ holds if and only if $X \geq d$. 
Let $\delta:= \frac{dk}{\Delta}-1$, so $d=(1+ \delta)p\Delta$. 
Then $\mathbb{P}(A_{v,c}) = \mathbb{P}(X \geq d) = \mathbb{P}( X \geq (1+\delta)p\Delta)$. 
Now
\begin{equation*} 
\frac{\delta^2p\Delta}{3}
= \frac13 \left( \frac{dk}{\Delta}-1\right)^2p\Delta
= \log(4k\Delta^2).
\end{equation*}
By \cref{Chernoff} with $n=\Delta$, 
\begin{equation*}
\mathbb{P}(A_{v,c})  \leq e^{-\delta^2 p\Delta/3}  = (4k\Delta^2)^{-1},
\end{equation*}
as claimed. By \cref{LLL}, with positive probability no event occurs, implying the desired partition exists. 
\end{proof}

%%%%%%%%%%%%%%%%%%%%%%%
\section{Proof of \cref{MaxDegree}}
\label{ProofMaxDegree}

Our proof works by considering sets of orderings with stronger properties than separation. We start with a lemma about complete graphs. 

\begin{lem}
\label{Loopy}
Let $G$ be the complete graph on $n$ vertices including loops. 
Then for some integer $p\leq 10 \log n$, there are linear orderings $<_1,\dots,<_p$ of $V(G)$, such that:
\begin{enumerate}[(1)]
\item every pair of disjoint edges $e,f\in E(G)$  are separated in some $<_i$, and
\item for every vertex $v\in V(G)$ and distinct vertices $u,w\in V(G)\setminus\{v\}$, for some $i\in[p]$ we have $u<_iv<_iw$ or $w<_iv<_iu$. 
\end{enumerate}
\end{lem}

\begin{proof}
Let $p:=\floor{10\log n}$. For $i\in[p]$, let $<_i$ be a random linear ordering of $V(G)$. 

Let $e$ and $f$ be edges in $G$ with no common endpoint. 
If neither $e$ nor $f$ are loops, then the probability that $e$ and $f$ are separated in $<_i$ is $\frac{1}{3}$. 
If $e$ is a loop and $f$ is a non-loop, then the probability that $e$ and $f$ are separated in $<_i$ is $\frac{2}{3}$. 
If both $e$ and $f$ are loops, then they are always separated in $<_i$. 
Thus the probability that $e$ and $f$ are not separated in $<_i$ is at most $\frac{2}{3}$. 
Hence the probability that (1) fails for $e$ and $f$ is at most $(\frac{2}{3})^p$. 

Now consider a vertex $v\in V(G)$ and distinct vertices $u,w\in V(G)\setminus\{v\}$. For each $i\in[p]$ the probability that  $u<_iv<_iw$ or $w<_iv<_iu$ is $\frac{1}{3}$. Hence the probability (2)  fails for every $i\in[p]$ is at most $(\frac{2}{3})^p$.

%\comment{AS: p4 Para 4.  We could (I think) improve on the union bound by using the Local Lemma again; but it would only gain a constant factor, so probably isn't worth it.}

By the union bound, the probability that both (1) and (2) fail is at most 
$\binom{|E(G)|}{2} (\frac{2}{3})^p + n \binom{n-2}{2} (\frac{2}{3})^p 
= ( \binom{n(n+1)/2}{2} + n \binom{n-2}{2}) (\frac{2}{3})^p 
< n^4 (\frac{2}{3})^p<1$. 
Thus there exists linear orderings $<_1,\dots,<_p$ such that (1) and (2) hold. 
\end{proof}

Note that we need $\Omega(\log n)$ orderings in \cref{Loopy} because of (2): if $p<\log_2(n-1) - 1$ then for any vertex $v$ and any set of $p$ orderings, there are distinct vertices $x,y$ are on the same side of $v$ in each of the orderings. 

%\smallskip
The following definition is a key to the proof of \cref{MaxDegree}. A representation $<_1,\dots,<_p$ of a graph $G$ is \emph{strongly separating} if:
\begin{enumerate}[(a)]
\item for all disjoint edges $vw,xy\in E(G)$, for some ordering $<_i$, we have $v,w<_ix,y$ or $x,y<_iv,w$, and
\item for every edge $vw\in E(G)$ and vertex $x \in V(G)\setminus\{v,w\}$, 
we have $x <_i v,w$ and $v,w <_j x$ for some $i,j\in[p]$.
\end{enumerate}
We define the \emph{strong separation dimension} of a graph $G$ to be the minimum number of linear orderings in a strongly separating representation of $G$.  Clearly the  separation dimension of a graph is at most its strong separation dimension, and it will be helpful to work with the latter.

\begin{lem}
\label{WeakStrong}
Every graph $G$ with maximum degree $\Delta$ has strong separation dimension at most the separation dimension of $G$ plus $2\Delta+2$. 
\end{lem}

\begin{proof}
Say $G$ has separation dimension $d$. 
By Vizing's Theorem, there is a partition $E_1,\dots,E_{\Delta+1}$ of $E(G)$ into matchings. 
Starting from a separating representation of $G$ in $d$ dimensions, we now add two orderings $<_i$ and $<'_i$ for each $i\in[\Delta+1]$. 
Say $E_i=\{v_1w_1,\dots,v_nw_n\}$. Let $<_i$ be $v_1,w_1,\dots,v_nw_n$ followed by $V(G)\setminus\{v_1,w_1,\dots,v_n,w_n\}$ in any ordering. Let $<'_i$ be the reverse of $<'_i$. 
Every edge $vw$ of $G$ is in some $E_i$. Since $v$ and $w$ are consecutive in $<_i$, for each vertex $x\in V(G)\setminus\{v,w\}$, we have 
$v,w <_i x$ and $x <'_i v,w$, or $v,w <'_i x$ and $x <_i v,w$.
Hence we have a strongly separating representation of $G$ with $d+2\Delta+2$ orderings in total.  
\end{proof}

\begin{lem}
\label{Components}
Let $G_1,\dots,G_k$ be the connected components of a graph $G$. 
For $a\in[k]$, let $p_a$ be the strong separation dimension of $G_a$. 
Then $G$ has strong separation dimension at most $\max\{p_1,\dots,p_k,2\}$. 
Moreover, there is such a representation such that in each ordering, 
$V(G_1) < V(G_2) < \dots < V(G_k)$ or $V(G_k) < V(G_{k-1}) < \dots < V(G_1)$.
\end{lem}

\begin{proof}
Let $p:=\max\{p_1,\dots,p_k,2\}$.
For $a\in[k]$, let $\{<^a_1,\dots,<^a_p\}$ be a strongly separating representation of $G_a$. 
%For $j\in[p-1]$, let $<_j$ be the ordering $<^1_j,\dots,<^k_j$ of $V(G)$. 
For $j\in[p-1]$, let $<_j$ be the ordering of $V(G)$ with $V(G_1)<_j\dots<_j V(G_k)$, where $V(G_a)$ is internally ordered according to $<^a_j$, for $a\in[k]$. 
Finally, let $<_p$ be the ordering of $V(G)$ with $V(G_k)<_p\dots<_p V(G_1)$, where $V(G_a)$ is internally ordered according to $<^a_p$, for $a\in[k]$. 
%Let $<_p$ be the ordering $<^k_p,\dots,<^1_p$ of $V(G)$. 
Thus $\{<_1,\dots,<_p\}$ is a representation of $G$, which we now show is strongly separating. 
Consider disjoint edges $vw,xy\in E(G)$. 
If $vw$ and $xy$ are in the same component, then (a) holds by assumption. 
Otherwise, $vw$ and $xy$ are in distinct components, implying that $v,w<_1 x,y$ or $x,y <_1 v,w$, 
and again (a) holds. 
Now consider an edge $vw\in E(G)$ and vertex $x \in V(G)\setminus\{v,w\}$.
If $vw$ and $x$ are in the same component, then (b) holds by assumption. 
So we may assume that $vw\in E(G_a)$ and $x\in V(G_b)$ for distinct $a,b\in[k]$. 
If $a<b$ then $v,w <_1 x$ and $x <_p v,w$. 
If $b<a$ then $v,w <_p x$ and $x <_1 v,w$. 
Thus (b) holds, and $\{<_1,\dots,<_p\}$ is strongly separating. 
\end{proof}

Note that every connected graph with at least three vertices has strong separation dimension at least 2, so \cref{Components} implies that for every graph $G$ with at least three vertices in some component, the strong separation dimension of $G$ equals the maximum strong separation dimension of the components of $G$. 

For a graph $G$ and disjoint sets $A,B\subseteq V(G)$, let $G[A,B]$ be the bipartite subgraph of $G$ with vertex set $A\cup B$ and edge set $\{vw\in E(G): v\in A, w\in B\}$.
%, called the  bipartite subgraph \emph{induced} by $A,B$. 

\begin{lem}
\label{partition}
Fix integers $s,t,k\geq 2$, where $k$ is even. 
Let $G$ be a graph, and let $V_1,\dots,V_k$ be a partition of $V(G)$, 
such that $G[V_i]$ has strong separation dimension at most $s$ for each $i\in[k]$, 
and $G[V_i,V_j]$ has strong separation dimension at most $t$ for all distinct $i,j\in[k]$. 
Then $G$ has strong separation dimension at most $2s+(k-1)t+20\log k$. 
\end{lem}

\begin{proof}
Let $G_0:=\bigcup_{i=1}^k G[V_i]$. Let $H$ be the complete graph with vertex set  $[k]$. Let $E_1,\dots,E_{k-1}$ be a partition of $E(H)$ into perfect matchings, which exists since $k$ is even. For $i\in[k-1]$, let $G_i :=\bigcup_{ab\in E_i} G[V_a,V_b]$. Note that $V(G_i)=V(G)$ for $i\in[0,k-1]$, and that $G=G_0\cup G_1 \cup \dots\cup G_{k-1}$. 

Since $s,t\geq 2$, by \cref{Components}, $G_0$ has strong separation dimension at most $s$, and $G_i$ has strong separation dimension at most $t$ for each $i\in[k-1]$. This gives $s+(k-1)t$ orderings of $V(G)$. Moreover, by \cref{Components}, for each of the $s$ orderings of $G_0$, we have $V_1 < \dots < V_k$ or $V_k  < \dots < V_1$. For each such ordering of $G_0$ of the form $V_1 < \dots < V_k$, add the \emph{extra} ordering $V_k < \dots < V_1$ to the representation of $G$. And for each such ordering of $G_0$ of the form $V_k  < \dots < V_1$, add the \emph{extra}  ordering $V_1 < \dots < V_k$ to the representation of $G$. In these extra orderings, each set $V_i$ inherits its ordering from the original. (So the extra ordering is not simply the reverse of the original.)\ This gives $2s+(k-1)t$ orderings of $V(G)$. 

For each $i\in[k]$, let $\overrightarrow{V_i}$ be an arbitrary linear ordering of $V_i$. 
Let $\overleftarrow{V_i}$ be the reverse ordering. 
Let $H^+$ be the complete graph on vertex set $[k]$ including loops. 
By \cref{Loopy}, for some $p\leq 10\log k$, there is a representation $\{<_1,\dots,<_p\}$ of $H^+$ such that:

\begin{enumerate}[(1)]
\item each pair of disjoint edges $e,f\in E(H^+)$  are separated in some $<_i$, and
\item  for every vertex $v\in V(H^+)$ and for all distinct vertices $u,w\in V(H^+)\setminus\{v\}$, for some $i\in[p]$ we have $u<_iv<_iw$ or $w<_iv<_iu$. 
\end{enumerate}

For each $i\in[p]$, introduce two orderings $<_i^+$ and $<_i^-$ of $V(G)$ constructed from $<_i$: in the first replace each vertex $i\in V(H^+)$ by $\overrightarrow{V_i}$, and in the second replace each vertex $i\in V(H^+)$ by $\overleftarrow{V_i}$. Together with the previous orderings, this gives a total of at most $2s+(k-1)t+20\log k$ orderings of $V(G)$. 

We now check that each pair of disjoint edges $vw$ and $xy$ in $G$ are separated in some ordering. 
Say $v\in V_i$, $w\in V_j$, $x\in V_a$ and $y\in V_b$. 

If $i=j$ and $a=b$, then $vw$ and $xy$ are both in $G_0$, and are thus separated in some ordering arising from $G_0$. So we may assume that $i \neq j$ or $a\neq b$. Without loss of generality, $i\neq j$. 

If $\{i,j\}=\{a,b\}$ then $ij\in E_\ell$ for some $\ell\in[k-1]$, implying $vw$ and $xy$ are both in $G_\ell$, and are thus separated in some ordering arising from $G_\ell$. So we may assume that $\{i,j\}\neq \{a,b\}$. Thus $ij$ and $ab$ are distinct edges of $H^+$, where $ab$ is possibly a loop. 

If $\{i,j\}\cap \{a,b\}=\emptyset $ then $ij$ and $ab$ are separated in some ordering $<_h$  arising from $H^+$, implying that $vw$ and $xy$ are also separated  (in both $<_h^+$ and $<_h^-$). So we may assume that $\{i,j\}\cap \{a,b\}\neq \emptyset$. Without loss of generality, $i=a$.

First suppose that $a=b$ ($=i$). Then $xy\in E(G_0)$ and $v\in V(G_0)$. Thus for some ordering $<_\alpha$ of $G_0$, we have $v <_\alpha x,y$. By construction, $V_j <_\alpha V_i$ or $V_i <_\alpha V_j$. If $V_j <_\alpha V_i$ then $w<_\alpha v<_\alpha  x,y$. Otherwise, $V_i <_\alpha V_j$. Then in the extra ordering associated with $<_\alpha$, we have $w<v<x,y$. In both cases, $vw$ and $xy$ are separated. 

So we may assume that $a\neq b$. Thus $j\neq b$, as otherwise $\{i,j\}=\{a,b\}$. By property (2) above, for some $r\in[p]$ we have $j<_r i<_r b$ or $b<_r i<_r j$. Without loss of generality, $j<_r i<_r b$. Since $v<x$ in $\overrightarrow{V_i}$ or in $\overleftarrow{V_i}$, in one of $<_r^+$ and $<_r^-$, we have $w < v < x < y$, implying $vw$ and $xy$ are separated. 

It remains to show that for every edge $vw\in E(G)$ and vertex $x \in V(G)\setminus\{v,w\}$, 
we have $x < v,w$  in some ordering and $v,w < x$ in another ordering. Since $vw\in E(G_i)$ for some $i\in[0,k-1]$, and $x\in V(G_i)$, this property holds by assumption. 
\end{proof}

We now prove \cref{MaxDegree}, which says that every graph with maximum degree $\Delta$ has separation dimension less than $20\Delta$. Recall that \citet{CMS11} proved the upper bound of $2\Delta (\ceil{\log_2 \log_2 \Delta} + 3) + 1$, which is less than $20\Delta$ if $\Delta \leq 2^{17}$. So it suffices to assume that $\Delta\geq 2^{17}$. In this case, to enable an inductive proof, we prove the following strengthening. 

\begin{lem}
\label{MaxDegreeProof}
For $\Delta\geq 2^{17}$, every graph with maximum degree at most $\Delta$ has strong separation dimension at most $20 \Delta(1 - \Delta^{-1/5} )$. 
\end{lem}

\begin{proof}
We proceed by induction on $\Delta$. In the base case, suppose that $2^{17} \leq \Delta\leq 2^{32}$. Let $G$ be a graph with maximum degree $\Delta$. By \cref{WeakStrong} and the result of \citet{CMS11} mentioned above, the strong separation dimension of $G$ is at most 
\begin{equation*}
2\Delta (\ceil{\log_2 \log_2 \Delta} + 4) + 3  
=18\Delta + 3  
\leq 20\Delta (1 - \Delta^{-1/5} ) .
\end{equation*}
%\begin{equation*}
%2\Delta (\ceil{\log_2 \log_2 \Delta} + 4) + 3 
%= \begin{cases}
%16\Delta + 3  \leq 22\Delta (1 - \Delta^{-1/5} ) & \text{ if } 2^8 < \Delta \leq 2^{16}\\
%18\Delta + 3  \leq 22\Delta (1 - \Delta^{-1/5} ) & \text{ if } 2^{16} < \Delta \leq 2^{32}\\
%20\Delta + 3  \leq 22\Delta (1 - \Delta^{-1/5} ) & \text{ if } 2^{32} < \Delta \leq 2^{32}.
%\end{cases}
%\end{equation*}
%To see that this upper bound is sufficient, note that:\\
%%if $3\leq  \Delta \leq 4$ then \eqref{easy} is $10\Delta + 3  \leq 22\Delta (1 - \Delta^{-1/5} )$; \\
%%if $4< \Delta \leq 2^4$ then \eqref{easy} is  $12\Delta + 3  \leq 22\Delta (1 - \Delta^{-1/5} )$; \\
%%if $2^4 < \Delta \leq 2^8$ then \eqref{easy} is  $14\Delta + 3  \leq 22\Delta (1 - \Delta^{-1/5} )$; \\
%if $2^8 < \Delta \leq 2^{16}$ then \eqref{easy} is  $16\Delta + 3  \leq 22\Delta (1 - \Delta^{-1/5} )$; \\
%if $2^{16} < \Delta \leq 2^{32}$ then \eqref{easy} is  $18\Delta + 3  \leq 22\Delta (1 - \Delta^{-1/5} )$. \\
%if $2^{32} < \Delta \leq 2^{64}$ then \eqref{easy} is  $20\Delta + 3  \leq 22\Delta (1 - \Delta^{-1/5} )$. 

So we may assume that $\Delta> 2^{32}$. 
Let $G$ be a graph with maximum degree $\Delta$. 
Let $k$ be the largest even integer at most $\Delta^{1/4}$. 
Let
\begin{equation*}
d:=  ( 1 + k^{-1} ) \,\frac{\Delta}{k}.
\end{equation*}
By \cref{Colour}, there is a partition $V_1,\dots,V_k$ of $V(G)$ such that  for every vertex $v\in V(G)$ and integer $i\in[k]$,
\begin{equation*}
|N_G(v)\cap V_i| < \frac{\Delta}{k} + \sqrt{\frac{3\Delta\log(4k \Delta^2)}{k} } < d,
\end{equation*}
where the final inequality holds since $k\leq \Delta^{1/4}$ and  $\Delta> 2^{32}$. 
Thus $G[V_i]$ and $G[V_i,V_j]$ have maximum degree at most $d$ for all distinct $i,j\in[k]$. 

Now $d\geq \frac{\Delta}{k} \geq \Delta^{3/4} \geq 2^{24}$ and $d<\Delta$. By induction, $G[V_i]$ and $G[V_i,V_j]$ both have strong separation dimension at most $20d ( 1 - d^{-1/5} )$ for all distinct $i,j\in[k]$. Since 
$20d ( 1 - d^{-1/5} ) \geq 2$, by \cref{partition}, $G$ has strong separation dimension at most 
$20 (k+1) d (1 - d^{-1/5} ) + 20 \log k$, which is at most 
$20 (k+2) d (1 - d^{-1/5} )$. 
All that remains is to prove that
\begin{equation}
\label{remains}
(k+2) d (1 - d^{-1/5} ) \leq \Delta (1 - \Delta^{-1/5} ).
\end{equation}
Suppose for the sake of contradiction that \eqref{remains} does not hold. 
Substituting for $d$ and since $k+4 \geq (k+2)(1+ k^{-1} )$, 
\begin{align*}
(k+4) \,\frac{\Delta}{k}\, (1 - d^{-1/5} ) \geq (k+2)  ( 1 + k^{-1} ) \, \frac{\Delta}{k}\, (1 - d^{-1/5} ) > \Delta ( 1 - \Delta^{-1/5}) .
\end{align*}
Thus
\begin{align*}
(1+4k^{-1} )  (1 - d^{-1/5} ) >   1 - \Delta^{-1/5} .
\end{align*}
Hence
\begin{align*}
4k^{-1}   +  \Delta^{-1/5} > (1+ 4k^{-1} )  d^{-1/5} > d^{-1/5}.
\end{align*}
Since $k\geq \frac45 \Delta^{1/4}$ and $d< \frac32 \Delta^{3/4}$, 
\begin{align*}
5\Delta^{-1/4}  + \Delta^{-1/5}    >  \left( \frac32 \Delta^{3/4} \right) ^{-1/5}  ,
\end{align*}
which is a contradiction since $\Delta > 2^{32}$. Hence \eqref{remains} holds, which completes the proof. 
\end{proof}

%%%%%%%%%%%%%%
\section{Proof of \cref{AvgDeg}}
\label{ProofAvgDeg}

This section shows that graphs with separation dimension 3 have bounded average degree. Much of the proof works in any dimension, so we present it in general. We include proofs of the following two folklore lemmas for completeness. 

\begin{lem}
\label{DropMinDegree}
Every graph with average degree at least $2d$ contains a subgraph with minimum degree at least $d$. 
\end{lem}

\begin{proof}
Deleting a vertex of degree less than $d$ maintains the property that the average degree is at least $2d$. 
Thus, repeatedly deleting vertices of degree less than $d$ produces a subgraph with average degree at least $2d$ and minimum degree at least $d$. 
%  m' = m - deg >= m - d
% n' = n-1
% avg deg 2m'/n' >= 2m/n
% <==   2(m-d)/(n-1) >= 2m/n
% <==   2(m-d)n  >= 2m(n-1)
% <==   2mn-2dn  >= 2mn-2m
% <==   2m >= 2dn
% <==   2m/n >= 2d true
\end{proof}

\begin{lem}
\label{DropBipartite}
Every graph with minimum degree at least $2d$ contains a bipartite spanning subgraph with minimum degree at least $d$. 
\end{lem}

\begin{proof}
For a partition $A,B$ of $V(G)$, let $e(A,B)$ be the number of edges between $A$ and $B$. 
Let $A,B$ be a partition of $V(G)$ maximising $e(A,B)$. If some vertex $v$ in $A$ has fewer than $d$ neighbours in $B$, then $v$ has more than $d$ neighbours in $A$, implying that $e(A\setminus\{v\}, B\cup \{v\})>e(A,B)$, which contradicts the choice of $A,B$. Thus each vertex in $A$ has at least $d$ neighbours in $B$, and by symmetry, every vertex in $B$ has at least $d$ neighbours in $A$. The result follows. 
\end{proof}

Let $G$ be a bipartite graph with bipartition $(A,B)$. 
A representation  $\{<_1,\dots,<_d\}$ of $G$ is \emph{consistent} if
for every edge $vw\in E(G)$ with $v\in A$ and $w\in B$, we have $v<_iw$ for all $i\in[d]$.
A representation  $\{<_1,\dots,<_d\}$ of $G$ is \emph{$A$-homogeneous} if
there are integers $a_1,\dots,a_d\in\{-1,+1\}$, such that for every vertex $v\in A$, 
there is a linear ordering $<_v$ of $N_G(v)$, with the property that for $i\in[d]$, 
\begin{itemize}
\item if $a_i=1$ then $N_G(v)$ is ordered in $<_i$ according to $<_v$, and
\item if $a_i=-1$ then $N_G(v)$ is ordered in $<_i$ according to $<'_v$, 
\end{itemize}
where $<'_v$ is the reverse of $<_v$.  The definition of \emph{$B$-homogeneous} is analogous. 
%A separating representation  $\{<_1,\dots,<_d\}$ of $G$ is \emph{$B$-homogeneous} if
%there are integers $b_1,\dots,b_d\in\{-1,+1\}$, such that for every vertex $w\in B$, 
%there is a linear ordering $<_w$ of $N_G(w)$, with the property that for $i\in[d]$, 
%\begin{itemize}
%\item if $b_i=1$ then $N_G(w)$ is ordered in $<_i$ according to $<_w$, and
%\item if $b_i=-1$ then $N_G(w)$ is ordered in $<_i$ according to $<'_w$, 
%\end{itemize}
%where $<'_v$ is the reverse of $<_v$. 
%A separating representation of $G$ is \emph{homogeneous} if it is both $A$-homogeneous and $B$-homogeneous. 

\begin{lem}
\label{ForceHomo}
Suppose that for some positive integers $d$ and $t$, there is a graph $G$ with average degree at least $2^{d+2}( 2^{d+1}t)^{2^{d-1}}$ and separation dimension at most $d$. Then there is a bipartite subgraph $G'$ of $G$ with bipartition $(A',B')$, with minimum degree at least $t$, such that $G'$ has a $d$-dimensional consistent separating representation that is $A'$-homogeneous or $B'$-homogeneous. 
\end{lem}

\begin{proof}
Let $\{<_1,\dots,<_d\}$ be a separating representation of $G$. By \cref{DropBipartite}, $G$ contains a bipartite spanning subgraph $G_1$ with average degree at least $2^{d+1}( 2^{d+1}t)^{2^{d-1}}$. Then $\{<_1,\dots,<_d\}$ is a separating representation of $G_1$. Let $(A_1,B_1)$ be the bipartition of $G_1$. 

For each edge $vw\in E(G_1)$ with $v\in A_1$ and $w\in B_1$, let $f(vw)=( f_1(vw), \dots, f_d(vw))$, where $f_i(vw):=1$ if $v<_iw$, and  $f_i(vw):=-1$ if $w<_i v$ (for $i\in[d]$). Since $f$ takes at most $2^d$ values, there is a set $E_2\subseteq E(G_1)$ with $f(vw)=f(xy)$ for all $vw,xy\in E_2$, and $|E_2|\geq |E(G_1)| /2^d$. Let $G_2$ be the spanning subgraph of $G_1$ with edge set $E_2$. Thus $G_2$ has average degree at least $2( 2^{d+1}t)^{2^{d-1}}$. For $i\in[d]$, if $f_i(vw)=-1$ for $vw\in E_2$, then replace $<_i$ by $<'_i$. Thus $\{<_1,\dots,<_d\}$ is a consistent separating representation of $G_2$. This property is maintained for all subgraphs of $G_2$. 

By \cref{DropMinDegree}, $G_2$ contains a subgraph $G_3$ with minimum degree at least $( 2^{d+1}t)^{2^{d-1}}$. Let $A_3:=A_2\cap V(G_3)$ and $B_3:= B_2 \cap V(G_3)$. Thus $(A_3,B_3)$ is a bipartition of $G_3$. Without loss of generality, $|A_3| \geq |B_3|$. 

For each vertex $v\in A_3$, by the Erd\H{o}s-Szekeres Theorem~\citep{ES35} applied $d-1$ times, there is a subset $M_v$ of $N_{G_3}(v)$ that is monotone with respect to $<_1$ in each ordering $<_2,\dots,<_d$,  and 
\begin{equation*}
|M_v| \geq ( \deg_{G_3}(v))^{1/2^{d-1}} \geq 2^{d+1}t.
\end{equation*} 
Let $g(v)=(g_2(v),\dots,g_d(v))$, where  $g_i(v):=1$ if $M_v$ is forward in $<_i$, and $g_i(v):=-1$ if $M_v$ is backward in $<_i$, for $i\in[2,d]$.  Since $g$ takes at most $2^{d-1}$ values, there is a subset $A_4$ of $A_3$ such that $g(v)=g(x)$ for all $v,x\in A_4$, and $|A_4| \geq |A_3| / 2^{d-1}$. Let $a_1:=1$ and for $i\in[2,d]$, let $a_i:=g_i(v)$ for $v\in A_4$. For $v\in A_4$, let $<_v$ be the ordering of $M_v$ in $<_1$. Let $B_4 := \bigcup_{v\in A_4} M_v$. Let $G_4$ be the bipartite subgraph with bipartition $(A_4,B_4)$, where $E(G_4):= \{vw:v\in A_4, w\in M_v\}$. By construction, $\{<_1,\dots,<_d\}$ is an $A_4$-homogeneous consistent separating representation of $G_4$. This property is maintained for all subgraphs of $G_4$. 

Note that every vertex in $A_4$ has degree at least $2^{d+1} t$ in $G_4$, and that 
\begin{equation*}
|V(G_4)| = |A_4| + |B_4| \leq |A_4| + |B_3| \leq |A_4| + |A_3| \leq (1 + 2^{d-1}) |A_4| \leq 2^d |A_4|.
\end{equation*} 
Hence $G_4$ has average degree  
\begin{equation*}
\frac{2|E(G_4)|}{|V(G_4)|} \geq \frac{2^{d+1} t |A_4| }{ 2^d |A_4|}  = 2t.
\end{equation*} 
By \cref{DropMinDegree}, $G_4$ contains a subgraph $G_5$ with minimum degree at least $t$. 
Let $A_5:=A_4\cap V(G_5)$. 
Then $\{<_1,\dots,<_d\}$ is an $A_5$-homogeneous consistent separating representation of $G_5$. 
\end{proof}

%%%%%%%%%

We now prove \cref{AvgDeg}. 

\begin{lem}
\label{SepDim3AvgDeg}
Every graph with separation dimension 3 has average degree less than $2^{29}$. 
\end{lem}

\begin{proof}
Suppose for the sake of contradiction that there is a graph with separation dimension 3 and average degree at least $2^{29}=2^{3+2}( 2^{3+1}4)^{2^{3-1}}$. By \cref{ForceHomo}, without loss of generality (possibly exchanging the roles of $A$ and $B$), there is a bipartite graph $G$ with bipartition $(A,B)$, with minimum degree at least $4$, such that $G$ has a $3$-dimensional $A$-homogeneous consistent separating representation $\{<_1,<_2,<_3\}$. Thus there are integers $a_1,a_2,a_3\in\{-1,+1\}$, such that for every vertex $v\in A$, there is a linear ordering $<_v$ of $N_G(v)$, with the property that for $i\in[3]$, 
\begin{itemize}
\item if $a_i=1$ then $N_G(v)$ is ordered in $<_i$ according to $<_v$, and
\item if $a_i=-1$ then $N_G(v)$ is ordered in $<_i$ according to $<'_v$.
\end{itemize}

%\comment{not needed?} First suppose that $G$ contains a 4-cycle $(a,b,c,d)$, where $a,c\in A$ and $b,d\in B$. By consistency, $a,c <_i b,d$ for each $i\in[3]$. In such an ordering no two disjoint edges are separated. Thus the 4-cycle has no consistent separating representation (in any dimension). So we may assume that $G$ contains no 4-cycle. 

By symmetry (since we may reverse all orders $<_v$), we may assume that at least two of $a_1,a_2,a_3$ are $+1$. Reordering leaves two cases: $a_1=a_2=a_3=1$, or $a_1=a_2=1$ and $a_3=-1$.
 
Case 1. $a_1=a_2=a_3=1$: Let $v$ be a vertex in $A$. Let $b,c$ be neighbours of $v$ with $b<_v c$. Since $a_1=a_2=a_3=1$, we have $v<_ib<_ic$ for each $i\in[3]$. Let $x$ be a neighbour of $b$ other than $v$ (which exists since $G$ has minimum degree at least 3). Then $vc$ and $bx$ are separated in no ordering, which is a contradiction. 

Case 2. $a_1=a_2=1$ and $a_3=-1$: For each vertex $v\in A$, \emph{mark} the rightmost edge incident with $v$ according to the ordering $<_v$ of $N_G(v)$. Since $G$ has at least $2|V(G)|$ edges and at most $|V(G)|$ edges are marked, $G$ contains a cycle $C$ of unmarked edges. As shown above, $C$ is not a 4-cycle. So $|C| \geq 6$. 

Let $v$ be the leftmost vertex in $C$ in $<_1$. 
Let $b$ and $c$ be the neighbours of $v$ in $C$. 
Without loss of generality, $b<_v c$. 
Since $a_1=a_2=1$ and $a_3=-1$, we have that $v <_1 b <_1 c$ and $v <_2 b <_2 c$ and $v <_3 c <_3 b$. 
Let $w$ be the neighbour of $b$ in $C$, such that $w\neq v$. 
%, which exists since $|C|\geq 6$. 
Note that $v,w\in A$ and $b,c\in B$. 
Since $b$ is between $v$ and $c$ in $<_1$ and $<_2$, 
the edges $vc$ and $wb$ are not separated in $<_1$ and $<_2$. 
Thus $vc$ and $wb$ are separated in $<_3$, implying
$v <_3 c <_3 w <_3 b$ by consistency. 
By the choice of $v$ and by consistency, $v <_1 w <_1 b <_1 c$. 
And by consistency, $v <_2 w <_2 b$ or $w <_2 v <_2 b$.

Let $b'$ be the rightmost neighbour of $w$ in $<_w$. 
Thus $wb'$ is marked. 
Since $w$ is between $v$ and $b$ in $<_1$ and $<_3$, 
the edges $vb$ and $wb'$ are not separated in $<_1$ and $<_3$.
Thus $vb$ and $wb'$ are separated in $<_2$. 
Since $a_2=+1$ and $b'$ is the rightmost neighbour of $w$ in $<_w$, we have $b <_2 b'$. 
Thus $v <_2 w <_2 b <_2 b'$ or $w <_2 v <_2 b <_2 b'$. 
In both cases, $vb$ and $wb'$ are not separated in $<_2$, which is a contradiction. 
\end{proof}

\citet{ABCMR18} state that it is open whether graphs with bounded separation dimension have bounded chromatic number. Since separation dimension is non-decreasing under taking subgraphs, \cref{SepDim3AvgDeg} implies:

\begin{cor}
\label{SepDim3Colour}
Every graph with separation dimension 3 is $2^{29}$-colourable.
\end{cor}

%\begin{thm}[???] 
%Let $V_1,V_2 $ be disjoint sets of vertices in a graph $G$, such that there is an edge between every $s_1$-subset of $V_1$ and every $s_2$-subset of $V_2$, then 
%$\pi(G) \geq \min \{ \log \frac{|V_1|}{s_1} , \log \frac{|V_2|}{s_2} \}$.
%\end{thm}
%
%\begin{cor}
%\begin{equation*}\pi(K_{n,n}) \geq \log n\end{equation*}
%\end{cor}

Recall that \citet{ABCMR18} proved that every $n$-vertex graph with separation dimension $s\geq 2$ has average degree $O(\log^{s-2} n)$. Their proof is by induction on $s$. Applying \cref{AvgDeg} in the base case leads to the following result:

\begin{cor}
For $s\geq 3$, every $n$-vertex graph with separation dimension $s$ has average degree $O(\log^{s-3} n)$. 
\end{cor}

For each $s\geq 4$, it remains open whether graphs of separation dimension at most $s$ satisfy analogues of \cref{SepDim3AvgDeg,SepDim3Colour}.

%\bibliographystyle{myNatbibStyle}
%\bibliography{myBibliography}

\begin{thebibliography}{16}
\providecommand{\natexlab}[1]{#1}
\providecommand{\url}[1]{\texttt{#1}}
\providecommand{\urlprefix}{}
\expandafter\ifx\csname urlstyle\endcsname\relax
  \providecommand{\doi}[1]{doi:\discretionary{}{}{}#1}\else
  \providecommand{\doi}{doi:\discretionary{}{}{}\begingroup
  \urlstyle{rm}\Url}\fi

\bibitem[{Alon et~al.(2015)Alon, Basavaraju, Chandran, Mathew, and
  Rajendraprasad}]{ABCMR15}
\textsc{Noga Alon, Manu Basavaraju, L.~Sunil Chandran, Rogers Mathew, and
  Deepak Rajendraprasad}.
\newblock Separation dimension of bounded degree graphs.
\newblock \emph{SIAM J. Discrete Math.}, 29(1):59--64, 2015.
\newblock \doi{10.1137/140973013}.
\newblock \msn{3295685}.

\bibitem[{Alon et~al.(2018)Alon, Basavaraju, Chandran, Mathew, and
  Rajendraprasad}]{ABCMR18}
\textsc{Noga Alon, Manu Basavaraju, L.~Sunil Chandran, Rogers Mathew, and
  Deepak Rajendraprasad}.
\newblock Separation dimension and sparsity.
\newblock \emph{J. Graph Theory}, 89(1):14--25, 2018.
\newblock \doi{10.1002/jgt.22236}.

\bibitem[{Basavaraju et~al.(2014)Basavaraju, Chandran, Golumbic, Mathew, and
  Rajendraprasad}]{BCGMR14}
\textsc{Manu Basavaraju, L.~Sunil Chandran, Martin~Charles Golumbic, Rogers
  Mathew, and Deepak Rajendraprasad}.
\newblock Boxicity and separation dimension.
\newblock In \emph{Graph-theoretic concepts in computer science}, vol. 8747 of
  \emph{Lecture Notes in Comput. Sci.}, pp. 81--92. Springer, Cham, 2014.
\newblock \doi{10.1007/978-3-319-12340-0\_7}.
\newblock \msn{3295860}.

\bibitem[{Basavaraju et~al.(2016)Basavaraju, Chandran, Golumbic, Mathew, and
  Rajendraprasad}]{BCGMR16}
\textsc{Manu Basavaraju, L.~Sunil Chandran, Martin~Charles Golumbic, Rogers
  Mathew, and Deepak Rajendraprasad}.
\newblock Separation dimension of graphs and hypergraphs.
\newblock \emph{Algorithmica}, 75(1):187--204, 2016.
\newblock \doi{10.1007/s00453-015-0050-6}.
\newblock \msn{3492062}.

\bibitem[{Bharathi et~al.(2017)Bharathi, De, and Lahiri}]{BDL17}
\textsc{Arpitha~P. Bharathi, Minati De, and Abhiruk Lahiri}.
\newblock Circular separation dimension of a subclass of planar graphs.
\newblock \emph{Discrete Math. Theor. Comput. Sci.}, 19(3):\#8, 2017.
\newblock \urlprefix\url{https://dmtcs.episciences.org/4031/}.

\bibitem[{Chandran et~al.(2011)Chandran, Mathew, and Sivadasan}]{CMS11}
\textsc{L.~Sunil Chandran, Rogers Mathew, and Naveen Sivadasan}.
\newblock Boxicity of line graphs.
\newblock \emph{Discrete Math.}, 311(21):2359--2367, 2011.
\newblock \doi{10.1016/j.disc.2011.06.005}.
\newblock \msn{2832135}.

\bibitem[{Erd\H{o}s and Szekeres(1935)}]{ES35}
\textsc{Paul Erd\H{o}s and George Szekeres}.
\newblock A combinatorial problem in geometry.
\newblock \emph{Compositio Math.}, 2:463--470, 1935.
\newblock \urlprefix\url{http://www.numdam.org/item?id=CM_1935__2__463_0}.

\bibitem[{Erd{\H{o}}s and Lov\'{a}sz(1975)}]{EL75}
\textsc{Paul Erd{\H{o}}s and L\'{a}szl\'{o} Lov\'{a}sz}.
\newblock Problems and results on {$3$}-chromatic hypergraphs and some related
  questions.
\newblock In \emph{Infinite and Finite Sets}, vol.~10 of \emph{Colloq. Math.
  Soc. J\'anos Bolyai}, pp. 609--627. North-Holland, 1975.
\newblock \urlprefix\url{https://www.renyi.hu/~p_erdos/1975-34.pdf}.
\newblock \msn{0382050}.

\bibitem[{F{\"u}redi and Kahn(1986)}]{FK86}
\textsc{Zolt{\'a}n F{\"u}redi and Jeff Kahn}.
\newblock On the dimensions of ordered sets of bounded degree.
\newblock \emph{Order}, 3(1):15--20, 1986.
\newblock \doi{10.1007/BF00403406}.
\newblock \msn{850394}.

\bibitem[{Hind et~al.(1997)Hind, Molloy, and Reed}]{HMR97}
\textsc{Hugh Hind, Michael Molloy, and Bruce Reed}.
\newblock Colouring a graph frugally.
\newblock \emph{Combinatorica}, 17(4):469--482, 1997.
\newblock \doi{10.1007/BF01195001}.
\newblock \msn{1645682}.

\bibitem[{Kang and M\"uller(2011)}]{KM11}
\textsc{Ross~J. Kang and Tobias M\"uller}.
\newblock Frugal, acyclic and star colourings of graphs.
\newblock \emph{Discrete Appl. Math.}, 159(16):1806--1814, 2011.
\newblock \doi{10.1016/j.dam.2010.05.008}.

\bibitem[{Loeb and West(2018)}]{LW18}
\textsc{Sarah~J. Loeb and Douglas~B. West}.
\newblock Fractional and circular separation dimension of graphs.
\newblock \emph{European J. Combin.}, 69:19--35, 2018.
\newblock \doi{10.1016/j.ejc.2017.09.001}.

\bibitem[{Mitzenmacher and Upfal(2005)}]{MU05}
\textsc{Michael Mitzenmacher and Eli Upfal}.
\newblock \emph{Probability and computing}.
\newblock Cambridge University Press, 2005.
\newblock \doi{10.1017/CBO9780511813603}.

\bibitem[{Molloy and Reed(2010)}]{MR10}
\textsc{Michael Molloy and Bruce Reed}.
\newblock Asymptotically optimal frugal colouring.
\newblock \emph{J. Combin. Theory Ser. B}, 100(2):226--246, 2010.
\newblock \doi{10.1016/j.jctb.2009.07.002}.

\bibitem[{Scott and Wood(2018)}]{SW18}
\textsc{Alex Scott and David~R. Wood}.
\newblock Better bounds for poset dimension and boxicity, 2018.
\newblock \arXiv{1804.03271}.

\bibitem[{Ziedan et~al.(2018)Ziedan, Rajendraprasad, Mathew, Golumbic, and
  Dusart}]{ZRMGD18}
\textsc{Emile Ziedan, Deepak Rajendraprasad, Rogers Mathew, Martin~Charles
  Golumbic, and J\'{e}r\'{e}mie Dusart}.
\newblock The induced separation dimension of a graph.
\newblock \emph{Algorithmica}, 80(10):2834--2848, 2018.
\newblock \doi{10.1007/s00453-017-0353-x}.

\end{thebibliography}

\def\soft#1{\leavevmode\setbox0=\hbox{h}\dimen7=\ht0\advance \dimen7
  by-1ex\relax\if t#1\relax\rlap{\raise.6\dimen7
  \hbox{\kern.3ex\char'47}}#1\relax\else\if T#1\relax
  \rlap{\raise.5\dimen7\hbox{\kern1.3ex\char'47}}#1\relax \else\if
  d#1\relax\rlap{\raise.5\dimen7\hbox{\kern.9ex \char'47}}#1\relax\else\if
  D#1\relax\rlap{\raise.5\dimen7 \hbox{\kern1.4ex\char'47}}#1\relax\else\if
  l#1\relax \rlap{\raise.5\dimen7\hbox{\kern.4ex\char'47}}#1\relax \else\if
  L#1\relax\rlap{\raise.5\dimen7\hbox{\kern.7ex
  \char'47}}#1\relax\else\message{accent \string\soft \space #1 not
  defined!}#1\relax\fi\fi\fi\fi\fi\fi}

\end{document}